\documentclass{birkjour}

 \newtheorem{thm}{Theorem}[section]
 \newtheorem{cor}[thm]{Corollary}
 \newtheorem{lem}[thm]{Lemma}
 \newtheorem{prop}[thm]{Proposition}
 \theoremstyle{definition}
 \newtheorem{defn}[thm]{Definition}
 \theoremstyle{remark}

 \numberwithin{equation}{section}

\begin{document}
%-------------------------------------------------------------------------
% editorial commands: to be inserted by the editorial office
%
%\firstpage{1}
%\volume{228}
%\DOI{003-0001}
%\seriesextraline{This is the Concrete Title of this Book\br H.E. R and S.T.C. W, Eds.}
%\firstpage{1}
\title[Canonical Connection on a Class of Riemannian Manifolds]
 {Canonical Connection on a Class of Rie\-man\-nian Almost Product Manifolds}

%----------Author 1
\author[D. Gribacheva]{Dobrinka
Gribacheva}
\address{%
University of Plovdiv \br Faculty of Mathematics and Informatics
\br 236 Bulgaria blvd\br Plovdiv 4003\br Bulgaria}

\email{dobrinka@uni-plovdiv.bg}

%----------Author 2
\author[D. Mekerov]{Dimitar Mekerov}

\address{%
University of Plovdiv \br Faculty of Mathematics and Informatics
\br 236 Bulgaria blvd\br Plovdiv 4003\br Bulgaria}

\email{mircho@uni-plovdiv.bg}

\thanks{This paper is partially supported by the project
NI11-FMI-004 of the Scientific Research Fund at University of
Plovdiv, Bulgaria.}



%\date{January 1, 2004}

\begin{abstract}
The canonical connection on a Riemannian almost product manifold
is an analogue to the Hermitian connection on an almost Hermitian
manifold. In this paper we consider the canonical connection on a
class of Riemannian almost product manifolds with non-integrable
almost product structure.
\end{abstract}

% MATH -------------------------------------------------------------------

\newcommand{\X}{\mathfrak{X}}
\newcommand{\B}{\mathcal{B}}
\newcommand{\s}{\mathfrak{S}}
\newcommand{\g}{\mathfrak{g}}
\newcommand{\W}{\mathcal{W}}
\newcommand{\Lgr}{\mathrm{L}}
\newcommand{\dd}{\mathrm{d}}

\newcommand{\n}{\nabla}
\newcommand{\nn}{\nabla'}
\newcommand{\lm}{\lambda}
\newcommand{\pd}{\partial}
\newcommand{\ddx}{\frac{\pd}{\pd x^i}}
\newcommand{\ddy}{\frac{\pd}{\pd y^i}}
\newcommand{\ddu}{\frac{\pd}{\pd u^i}}
\newcommand{\ddv}{\frac{\pd}{\pd v^i}}

\newcommand{\diag}{\mathrm{diag}}
\newcommand{\End}{\mathrm{End}}
\newcommand{\im}{\mathrm{Im}}
\newcommand{\id}{\mathrm{id}}

\newcommand{\ie}{i.e.}
\newfont{\w}{msbm9 scaled\magstep1}
\def\R{\mbox{\w R}}
\newcommand{\norm}[1]{\left\Vert#1\right\Vert ^2}
\newcommand{\nN}{\norm{N}}
\newcommand{\nP}{\norm{\nabla P}}
\newcommand{\tr}{{\rm tr}}

\newcommand{\nJ}[1]{\norm{\nabla J_{#1}}}
\newcommand{\thmref}[1]{Theorem~\ref{#1}}
\newcommand{\propref}[1]{Proposition~\ref{#1}}
\newcommand{\secref}[1]{\S\ref{#1}}
\newcommand{\lemref}[1]{Lemma~\ref{#1}}
\newcommand{\corref}[1]{Corollary~\ref{#1}}
\newcommand{\dfnref}[1]{Definition~\ref{#1}}
%\newcommand{\eqref}[1]{(\ref{#1})}

%\frenchspacing

\newcommand{\ee}{\end{equation}}
\newcommand{\be}[1]{\begin{equation}\label{#1}}

%%% ----------------------------------------------------------------------
\maketitle
%%% ----------------------------------------------------------------------
%\tableofcontents
%%%%%%%%%%%%%%%%%%%%%%%%%%%%%%%%%%%%%%%%%%%%%%%%%%%%%%%%%%%%%%%%%%%%%%%%%%%0

\section{Introduction}

On an Hermitian manifold $(M,J,g)$ there exists an unique linear
connection $D$ with a torsion  $T$ such that $DJ=Dg=0$ and
$T(x,Jy)=T(Jx,y)$ for all vector fields $x$, $y$ on $M$. This is
the Hermitian connection of the manifold
(\cite{Li-1,Li-2,GrayBNV}). The group of the conformal
transformations of the metric $g$ generates the conformal group of
the transformations of $D$.

Analogously to the Hermitian connection on an almost Hermitian
manifold, V.~Mihova in \cite{Mi} define on a Riemannian almost
product manifold $(M,P,g)$ a natural con\-nection $\n'$ (i.e.
$\n'P=\n'g=0$) with torsion $T$ satisfying
$T(x,y,z)\allowbreak +T(y,z,x)+T(Px,y,Pz)+T(y,Pz,Px)=0$. This
connection is called canonical and it is proved that it is unique
on $(M,P,g)$.

The systematic development of the theory of Riemannian almost
product manifolds was started by K. Yano \cite{Ya}. In \cite{Nav}
A.~M.~Naveira gives a classification of these manifolds with
respect to the covariant differentiation of the almost product
structure. Having in mind the results in \cite{Nav}, M.~Staikova
and K.~Gribachev give in \cite{Sta-Gri} a classification of the
Riemannian almost product manifolds $(M,P,g)$ with $\tr{P}=0$. The
intersection of the basic classes in this classification is the
class $\W_0$ of the Riemannian $P$-manifolds determined by $\n
P=0$, where $\n$ is the Levi-Civita connection.

In the present work we consider the canonical connection on the
manifolds of the class $\W_3$ from the Staikova-Gribachev
classification which we called Riemannian almost product
$\W_3$-manifolds. The class $\W_3$ is the only basic class where
any manifold $(M,P,g)\notin\W_0$ has a non-integrable almost
product structure $P$. This class is an object of interest in
\cite{Mek1}.

In Section 2 we give some necessary facts about the class $\W_3$.
We introduce the  notion of a Riemannian $P$-tensor $L$
which is curvature-like and has the property
$L(x,y,Pz,Pw)=L(x,y,z,w)$. This tensor is an analogue of the
K\"ahler tensor in Hermitian geometry.

In Section 3 we recall facts about the natural connections (\ie{}
the connections preserving $P$ and $g$) with torsion on Riemannian
almost product manifolds $(M,P,g)$. We find conditions for the
torsion of such a connection when $(M,P,g)$ is a Riemannian almost
product $\W_3$-manifold.

In Section 4 we consider the canonical connection $\n'$ on a
Riemannian almost product $\W_3$-manifold  $(M,P,g)$. We find properties of the
torsion of $\n'$ as well as the exact
expression of $\n'$. Let us point out the result in
\thmref{thm-4.4}, where the tensor norm $\nP$ is given in terms of
the scalar curvatures of $\n$ and $\n'$. Due to the definitness of
the metric $g$, the important \corref{cor-4.5} is obtained.

In Section 5 we establish properties of a Riemannian almost
product $\W_3$-mani\-fold for which the curvature tensor of the canonical connection
is a  Riemannian $P$-tensor.

In Section 6 we study the case of canonical connection $\n'$ with
parallel torsion $T$ on  a Riemannian almost product $\W_3$-manifold.
The relation between the curvature tensors of $\n$ and $\n'$ is
found. An important result is given in \thmref{thm-6.4}, where
we prove that a necessary condition for the curvature tensor of $\n'$ to be
a Riemannian $P$-tensor is the parallelism of the torsion $T$ with respect to $\n'$.

In Section 7 we consider a 4-dimensional Riemannian
almost product $\W_3$-manifold $(G,P,g)$, where $G$ is a Lie
group. The properties of Riemannian almost product manifolds
$(G,P,g)$   are expressed in terms of the commutators in the
corresponding Lie algebra. At first, we find some geometrical
characteristics of the manifold $(G,P,g)$. After that, we
interpret theoretical results obtained in this paper in accordance
with the canonical connection on $(G,P,g)$.

%%%%%%%%%%%%%%%%%%%%%%%%%%%%%%%%%%%%%%%%%%%%%%%%%%%%%%%%1

\section{Riemannian almost product $\W_3$-manifolds}

Let $(M,P,g)$ be a \emph{Riemannian almost product manifold},
\ie{} a differentiable manifold $M$ with a tensor field $P$ of
type $(1,1)$ and a Riemannian metric $g$ such that
\begin{equation}\label{Pg}
    P^2x=x,\quad g(Px,Py)=g(x,y)
\end{equation}
for arbitrary $x$, $y$ of the algebra $\X(M)$ of the smooth vector
fields on $M$. Obviously $g(Px,y)=g(x,Py)$.

Further $x,y,z,w$ will stand for arbitrary elements of $\X(M)$ or
vectors in the tangent space $T_pM$ at $p\in M$.

In this work we consider Riemannian almost product manifolds with
$\tr{P}=0$. In this case $(M,P,g)$ is an even-dimensional
manifold. If $\dim{M}=2n$ then the \emph{associated metric}
$\tilde{g}$ of $g$, determined by $\tilde{g}(x,y)=g(x,Py)$, is an
indefinite metric of signature $(n,n)$. Since
$\tilde{g}(Px,Py)=\tilde{g}(x,y)$, the manifold $(M,P,\tilde{g})$
is a \emph{pseudo-Riemannian almost product manifold}.

The classification in \cite{Sta-Gri} of Riemannian almost product
manifolds with $\tr{P}=0$ is made with respect to the tensor  $F$
of type (0,3), defined by
\begin{equation}\label{2}
F(x,y,z)=g\left(\left(\nabla_x P\right)y,z\right),
\end{equation}
where $\nabla$ is the Levi-Civita connection of $g$. The tensor
$F$ has the following properties:
\begin{equation}\label{3}
    F(x,y,z)=F(x,z,y)=-F(x,Py,Pz),\; F(x,y,Pz)=-F(x,Py,z).
\end{equation}
The basic classes of the classification in \cite{Sta-Gri} are
$\W_1$, $\W_2$ and $\W_3$. Their intersection is the class $\W_0$
of the \emph{Riemannian $P$-manifolds}, determined by the
condition $F=0$ or equivalently $\n P=0$. In the classification
there are include the classes $\W_1\oplus\W_2$, $\W_1\oplus\W_3$,
$\W_2\oplus\W_3$ and the class $\W_1\oplus\W_2\oplus\W_3$ of all
Riemannian almost product manifolds.

In the present work we consider the manifolds from the class
$\W_3$. This class is determined by the condition
\begin{equation}\label{sigma}
    \mathop{\s}_{x,y,z} F(x,y,z)=0,
\end{equation}
where $\mathop{\s}_{x,y,z}$ is the cyclic sum by $x, y, z$. This
is the only class of the basic classes $\W_1$, $\W_2$ and $\W_3$,
where each manifold (which is not a Riemannian $P$-manifold) has a
\emph{nonintegrable almost product structure $P$}, i.e. the
Nijenhuis tensor $N$, determined by
\[%begin{equation}\label{4'}
    N(x,y)=\left(\nabla_x P\right)Py-\left(\nabla_y P\right)Px
    +\left(\nabla_{Px} P\right)y
    -\left(\nabla_{Py} P\right)x,
\]%end{equation}
is non-zero.

In \cite{Sta-Gri} it is introduced an \emph{associated tensor}
$N^*$ by
\[%begin{equation}\label{4''}
    N^*(x,y)=\left(\nabla_x P\right)Py+\left(\nabla_{Px} P\right)y
    +\left(\nabla_y P\right)Px+\left(\nabla_{Py} P\right)x.
\]%end{equation}
It is proved that the condition \eqref{sigma} is equivalent to
$N^*(x,y)=0$.

Further,  manifolds of the class $\W_3$ we call \emph{Riemannian
almost product $\W_3$-manifolds}.

As it is known the curvature tensor $R$ of a Riemannian manifold
with metric $g$ is determined by $
    R(x,y)z=\nabla_x \nabla_y z - \nabla_y \nabla_x z -
    \nabla_{[x,y]}z
$ and the corresponding $(0,4)$-tensor is defined as
follows $
    R(x,y,z,w)=g(R(x,y)z,w).
$

Let $(M,P,g)$ be a Riemannian almost product manifold and
$\{e_i\}$ be a basis of $T_pM$. Let the components of the inverse
matrix of $g$ with respect to $\{e_i\}$ be $g^{ij}$. Then the
quantities $\rho$ and $\tau$, determined by
$\rho(y,z)=g^{ij}R(e_i,y,z,e_j)$ and $\tau=g^{ij}\rho(e_i,e_j)$,
are the Ricci tensor and the scalar curvature for $\n$,
respectively.

The \emph{square norm} of $\nabla P$ is defined by
\begin{equation}\label{5}
\nP=g^{ij}g^{ks}g\left(\left(\nabla_{e_i}P\right)e_k,\left(\nabla_{e_j}P\right)e_s\right).
\end{equation}
Obviously $\nP=0$ if and only if $(M,P,g)$ is a Riemannian
$P$-manifold.

%%%%%%%%%%%%%%%%%%%%%%%%%%%%%%%%%%%%%%%%%%%%%%%%%%%%%%%%%%%%%%%%%%%%%%%%%%%%
A tensor $L$ of type (0,4) with pro\-per\-ties%
\be{2.11}%
L(x,y,z,w)=-L(y,x,z,w)=-L(x,y,w,z),
\ee %
\be{2.12} \mathop{\s} \limits_{x,y,z} L(x,y,z,w)=0 \quad
\textit{(the first Bianchi identity),}
\ee %
\begin{equation}\label{2.13}
L(x,y,Pz,Pw)=L(x,y,z,w),
\end{equation}
is called a \emph{Riemannian $P$-tensor}.

%%%%%%%%%%%%%%%%%%%                3

\section{Natural connection on Riemannian almost product manifolds}

The linear connections in our investigations have a torsion.

Let $\nn$ be a linear connection determined by $\nn_x y=\n_x y+Q(x,y)$, where $Q$ is a (1,2)-tensor.
The torsion (1,2)-tensor $T$ is determined by
$T(x,y)=\nn_x y-\nn_y x-[x,y]$.
The corresponding (0,3)-tensors are defined by
\begin{equation}\label{3.1}
    Q(x,y,z)=g(Q(x,y),z), \quad T(x,y,z)=g(T(x,y),z).
\end{equation}
The symmetry of the Levi-Civita connection implies
\begin{equation}\label{3.2}
    T(x,y)=Q(x,y)-Q(y,x),
\end{equation}
\[%begin{equation}\label{3.3}
    T(x,y)=-T(y,x).
\]%end{equation}

A partial decomposition of the space $\mathcal{T}$ of the torsion
tensors $T$ of type (0,3) is valid on a Riemannian  almost product
manifold $(M,P,g)$:
$\mathcal{T}=\mathcal{T}_1\oplus\mathcal{T}_2\oplus\mathcal{T}_3\oplus\mathcal{T}_4$,
where $\mathcal{T}_i$ $(i=1,2,3,4)$ are invariant orthogonal
subspaces \cite{Mi}. For the projection operators $p_i$ of
$\mathcal{T}$ in $\mathcal{T}_i$ is established:
\begin{equation*}
  \begin{split}
&
    p_1(x,y,z)=\frac{1}{8}\bigl\{2T(x,y,z)-T(y,z,x)-T(z,x,y)-T(Pz,x,Py)\bigr.\\[4pt]
& \phantom{p_1(x,y,z)=\frac{1}{8}}
    +T(Py,z,Px)+T(z,Px,Py)-2T(Px,Py,z)\\[4pt]
& \phantom{p_1(x,y,z)=\frac{1}{8}}
    +T(Py,Pz,x)+T(Pz,Px,y)-T(y,Pz,Px)\bigr\},\\[4pt]
\end{split}
\end{equation*}
\begin{equation*}
  \begin{split}
&
    p_2(x,y,z)=\frac{1}{8}\bigl\{2T(x,y,z)+T(y,z,x)+T(z,x,y)+T(Pz,x,Py)\bigr.\\[4pt]
&  \phantom{p_2(x,y,z)=\frac{1}{8}}
    -T(Py,z,Px)-T(z,Px,Py)-2T(Px,Py,z)\\[4pt]
& \phantom{p_2(x,y,z)=\frac{1}{8}}
    -T(Py,Pz,x)-T(Pz,Px,y)+T(y,Pz,Px)\bigr\},\\[4pt]
  &
    p_3(x,y,z)=\frac{1}{4}\bigl\{T(x,y,z)+T(Px,Py,z)-T(Px,y,Pz)\bigr.\\[4pt]
&
 \phantom{p_3(x,y,z)=\frac{1}{4}}
    \bigl.-T(x,Py,Pz)\bigr\},\\[4pt]
  &
    p_4(x,y,z)=\frac{1}{4}\bigl\{T(x,y,z)+T(Px,Py,z)+T(Px,y,Pz)\bigr.\\[4pt]
& \phantom{p_1(x,y,z)=\frac{1}{4}}
    \bigl.+T(x,Py,Pz)\bigr\}.
  \end{split}
\end{equation*}

\begin{defn}[\cite{Mi}]\label{dfn-3.1'}
A linear connection $\nn$ on a Riemannian almost product manifold
$(M,P,g)$ is called a \emph{natural connection} if $\nn P=\nn g=0$
(or equi\-va\-lently $\nn g=\nn \tilde{g}=0$).
\end{defn}
If $\nn$ is a linear connection with a (0,3)-tensor $Q$ on a Riemannian almost product
manifold, then it is  a natural connection if and only if the
following conditions are valid \cite{Mi}:
\begin{equation}\label{3.5}
    F(x,y,z)=Q(x,y,Pz)-Q(x,Py,z),
\end{equation}
\begin{equation}\label{3.6}
    Q(x,y,z)=-Q(x,z,y).
\end{equation}

Let $\Phi$ be the (0,3)-tensor determined by
\[%begin{equation}\label{3.7}
    \Phi(x,y,z)=g\left(\widetilde{\nabla}_x y - \n_x y,z \right),
\]%end{equation}
where $\widetilde{\nabla}$ is the Levi-Civita connection of the
associated metric $\tilde{g}$.

\begin{thm}[\cite{Mi}]\label{t-3.1}
A linear connection with the torsion  $T$ on a Riemannian almost
product manifold $(M,P,g)$ is natural if and only if
\begin{equation}\label{3.9}
  \begin{array}{l}
    4p_1(x,y,z)=-\Phi(x,y,z)+\Phi(y,z,x)-\Phi(x,Py,Pz)\\[4pt]
    \phantom{4p_1(x,y,z)=}-\Phi(y,Pz,Px)+2\Phi(z,Px,Py),
  \end{array}
\end{equation}
\begin{equation}\label{3.8}
    4p_3(x,y,z)=-g(N(x,y),z)=-2\left\{\Phi(z,Px,Py)+\Phi(z,x,y)\right\}.
\end{equation}
\end{thm}

\begin{prop}\label{prop-3.3}
    For the torsion  $T$  of a natural connection on a Riemannian almost product  $\W_3$-manifold
    $(M,P,g)\notin\W_0$, the following properties
    are valid
    \[%begin{equation}\label{3.12}
        p_1=0,\qquad p_3\neq 0.
    \]%end{equation}
\end{prop}
\begin{proof} In \cite{Sta-Gri} it is proved that the
both basic tensors $F$ and $\Phi$ on a Riemannian almost product
manifold $(M,P,g)$ are related as follows:
\begin{equation}\label{1.14'''}
    \Phi(x,y,z)=\frac{1}{2}\{-F(Pz,x,y)+F(x,y,Pz)+F(y,Pz,x)\}.
\end{equation}

If $(M,P,g)$ is a Riemannian almost product $\W_3$-manifold then
\eqref{3}, \eqref{sigma} and \eqref{1.14'''} imply
\begin{equation}\label{3.11}
    \Phi(x,y,z)=-F(Pz,x,y).
\end{equation}

From equalities \eqref{3.11}, \eqref{3.9}, \eqref{3},
\eqref{sigma} we get $p_1=0$. Since $N\neq 0$ for a Riemannian
almost product $\W_3$-manifold $(M,P,g)\notin \W_0$, then
\eqref{3.8} implies $p_3\neq 0$. \end{proof}

%%%%%%%%%%%%%%%%                   4

\section{Canonical connection on Riemannian almost product $\W_3$-ma\-ni\-folds }

\begin{defn}[\cite{Mi}]\label{dfn-4.1}
A natural connection with torsion  $T$ on %
a Riemannian almost product manifold  %
$(M,P,g)$ is called a \emph{canonical connection} if
\begin{equation}\label{4.1}
    T(x,y,z)+T(y,z,x)+T(Px,y,Pz)+T(y,Pz,Px)=0.
\end{equation}
\end{defn}

In \cite{Mi} it is shown that \eqref{4.1} is equivalent to the
condition
\begin{equation}\label{4.2}
        p_2=p_4=0,
\end{equation}
\ie{} to the condition $T\in\mathcal{T}_1\oplus\mathcal{T}_3$. The
same paper shows that on every Riemannian almost product manifold
$(M,P,g)$ there exists an unique canonical connection $\nn$, and
it is determined by
\begin{equation}\label{4.3}
    g(\nn_x y,z)=g(\n_x
    y,z)+\frac{1}{4}\left\{\Phi(x,y,z)-2\Phi(z,x,y)-\Phi(x,Py,Pz)\right\}.
\end{equation}
For the torsion $T$ of this connection it is valid
\begin{equation}\label{4.3'}
\begin{split}
    T(x,y,z)=\frac{1}{4}\bigl\{\Phi(y,z,x)&-\Phi(z,x,y)\\[4pt]
&+\Phi(y,Pz,Px)+\Phi(Pz,x,Py)\bigr\}.
\end{split}
\end{equation}

\begin{prop}\label{prop-4.1'}
    Let $T$ be the torsion of the canonical connection $\n'$ on a Riemannian almost product
    $\W_3$-manifold $(M,P,g)$. Then $T$ has the properties
\begin{equation}\label{4.4}
    T(Px,y)=-PT(x,y),
\end{equation}
\begin{equation}\label{4.5}
    %T(x,y,z)=-T(y,x,z),\quad
T(Px,y,z)=T(x,Py,z)=-T(x,y,Pz),
\end{equation}
\begin{equation}\label{30''}
\begin{split}
T\bigl(T(Px,Py),z,w\bigr)&=T\bigl(T(Px,y),z,Pw\bigr)
\\[4pt]
&=T\bigl(T(x,Py),z,Pw\bigr)=T\bigl(T(x,y),z,w\bigr),
\end{split}
\end{equation}
\begin{equation}\label{4.8'}
T=p_3, \; \text{\ie{}}\;
T\in\mathcal{T}_3
\end{equation}
and $\n'$ is determined by
\begin{equation}\label{4.7}
    \nn_x y=\n_x y+\frac{1}{4}\left\{-\left(\n_y P\right)Px
    +\left(\n_{Py} P\right)x-2\left(\n_{x} P\right)Py\right\}.
\end{equation}
\end{prop}
\begin{proof} By virtue of \eqref{4.3'}, \eqref{3.11},
\eqref{3} and \eqref{sigma} we obtain \eqref{4.4}. Then,
\eqref{4.4}, \eqref{3.1} and \eqref{Pg} imply \eqref{4.5} and
\eqref{30''}. From \propref{prop-3.3} and condition \eqref{4.2} we obtain
immediately \eqref{4.8'}.
Equalities \eqref{3.11} and \eqref{4.3} imply \eqref{4.7}.
\end{proof}

Let $\nn$ be the canonical connection on a Riemannian almost
product $\W_3$-manifold  $(M,P,g)$. According to \eqref{4.7}, for
the tensor $Q$  and the
torsion $T$ of $\n'$ we have
\begin{equation}\label{4.8}
    Q(x,y)=\frac{1}{4}\left\{-\left(\n_y P\right)Px+\left(\n_{Py} P\right)x-2\left(\n_{x}
    P\right)Py\right\},
\end{equation}
\begin{equation}\label{32'}
    T(x,y)=-\frac{1}{2}\left\{\left(\n_x P\right)Py+\left(\n_{Px} P\right)y\right\}.
\end{equation}

Hence, having in mind \eqref{32'}, \eqref{2} and \eqref{3.1}, we
obtain
\begin{equation}\label{4.9}
    T(x,y,z)=-\frac{1}{2}\left\{F(x,Py,z)+F(Px,y,z)\right\}.
\end{equation}
Substituting $y\leftrightarrow z$ into the above, according to
\eqref{3}, we get
\begin{equation}\label{33'}
    T(x,z,y)=\frac{1}{2}\left\{F(x,Py,z)-F(Px,y,z)\right\}.
\end{equation}
Subtracting \eqref{33'}  from \eqref{4.9} and replacing $y$ with
$Py$ in the result, we have
\begin{equation}\label{4.10}
    F(x,y,z)=T(x,z,Py)-T(x,Py,z).
\end{equation}

Equalities \eqref{4.8}, \eqref{3.1} and \eqref{2} imply
\begin{equation}\label{4.11}
    Q(x,y,z)=-\frac{1}{4}\left\{F(y,Px,z)-F(Py,x,z)+2F(x,Py,z)\right\}.
\end{equation}
Hence, because of \eqref{3} and \eqref{sigma}, we conclude that
\begin{equation}\label{4.12}
    Q(x,y,z)=-Q(y,x,z)-F(Pz,x,y).
\end{equation}

\begin{thm}\label{thm-4.4}
    Let $\tau'$ and $\tau$ be the scalar curvatures for the canonical connection $\nn$ and the Levi-Civita connection $\n$, respectively,
    on a Riemannian almost product $\W_3$-manifold $(M,P,g)$.
     Then
    \begin{equation}\label{4.13}
       \nP=8(\tau'-\tau).
    \end{equation}
\end{thm}
\begin{proof} According to \eqref{Pg} and \eqref{3},
for a Riemannian almost product manifold  we have
$g^{ij}F(Pz,e_i,e_j)=0$. Then, from \eqref{4.12}, after
contraction by $x=e_i$, $y=e_j$, we obtain
\begin{equation}\label{4.14}
    g^{ij}Q(e_i,e_j,z)=0.
\end{equation}
Because of $\n g^{ij}=0$  and \eqref{4.14}, we get
\begin{equation}\label{4.15}
    g^{ij}\left(\n_x Q\right)(e_i,e_j,z)=0.
\end{equation}

It is known that for the curvature tensors $R'$ and $R$ of $\nn$
and $\n$, respectively,  the following  is valid:
\[
\begin{split}
&R'(x,y,z,w)=R(x,y,z,w)+ \left(\n_x Q\right)(y,z,w)-\left(\n_y
Q\right)(x,z,w)
\\[4pt]
&\phantom{K(x,y,z,w)=R(x,y,z,w)}+Q\bigl(x,Q(y,z),w\bigr)-Q\bigl(y,Q(x,z),w\bigr).
\end{split}
\]
Then from \eqref{3.6} and \eqref{3.1} it follows that
\begin{equation}\label{4.16}
\begin{split}
&R'(x,y,z,w)=R(x,y,z,w)+ \left(\n_x Q\right)(y,z,w)-\left(\n_y
Q\right)(x,z,w)
\\[4pt]
&\phantom{K(x,y,z,w)=}-g\bigl(Q(x,w),Q(y,z)\bigr)+g\bigl(Q(y,w),Q(x,z)\bigr)
\end{split}
\end{equation}
for a Riemannian almost product manifold  $(M,P,g)$. Using a
contraction by $x=e_i$, $w=e_j$ in \eqref{4.16} and combining
\eqref{3.6}, \eqref{4.14} and \eqref{4.15}, we find that the Ricci
tensors $\rho'$ and $\rho$ for $\nn$ and $\n$ satisfy
\begin{equation}\label{4.17}
    \rho'(y,z)=\rho(y,z)+g^{ij}\left(\n_{e_i}
    Q\right)(y,z,e_j)+g^{ij}g\bigl(Q(y,e_j),Q(e_i,z)\bigr).
\end{equation}

Similarly, after a contraction by $y=e_k$, $z=e_s$ in \eqref{4.17}
and according to \eqref{4.15}, we obtain
\begin{equation}\label{4.18}
    \tau'=\tau+g^{ij}g^{ks}g\bigl(Q(e_k,e_j),Q(e_i,e_s)\bigr)
\end{equation}
for the scalar curvatures $\tau'$ and $\tau$ for $\nn$ and $\n$.

The equalities \eqref{4.18} and \eqref{4.8} imply on a Riemannian
$\W_3$-manifold $(M,P,g)$ the following equality
\begin{equation}\label{4.19}
    g^{ij}g^{ks}g\bigl(Q(e_k,e_j),Q(e_i,e_s)\bigr) =\frac{1}{16}g^{ij}g^{ks}g(A_{jk},A_{si})
\end{equation}
where
\[
A_{jk}=-\left(\n_{e_j}P\right)Pe_k+\left(\n_{Pe_j}P\right)e_k-2\left(\n_{e_k}P\right)Pe_j.
\]
From \eqref{4.19}, \eqref{Pg} and \eqref{5} we get
\[
g^{ij}g^{ks}g\bigl(Q(e_k,e_j),Q(e_i,e_s)\bigr)=\frac{1}{8}\nP.
\]
The last equality and \eqref{4.18} imply \eqref{4.13}.
\end{proof}

\begin{cor}\label{cor-4.5}
A Riemannian almost product  $\W_3$-manifold is a Riemannian
$P$-mani\-fold if and only if the scalar curvatures for the
canonical connection and the Levi-Civita connection are equal.

\end{cor}

%%%%%%%%%%%%%%%%%%%      5

\section{Canonical connection
on a Riemannian almost product $\W_3$-manifold whose curvature tensor is a Rie\-man\-nian $P$-ten\-sor}

It is known (\cite{KoNo}) that for every linear connection $\nn$
on a Riemannian manifold $(M,g)$ with a torsion $T$ and a
curvature tensor $R'$ the following equality (the first Bianchi
identity) is valid
\[%begin{equation}\label{44'}
\mathop{\s} \limits_{x,y,z} R'(x,y)z=\mathop{\s} \limits_{x,y,z}
\bigl\{\left(\nn_{x} T\right)(y,z)+T(T(x,y),z)\bigr\}.
\]%end{equation}

Let $\nn$ is a natural connection on a Riemannian almost product
manifold $(M,P,g)$. Then the latter equality and $\nn g=0$ imply
\begin{equation}\label{44''}
\mathop{\s} \limits_{x,y,z} R'(x,y,z,w)= \mathop{\s}
\limits_{x,y,z} \bigl\{\left(\nn_{x}
T\right)(y,z,w)+T(T(x,y),z,w)\bigr\}.
\end{equation}

For the curvature tensor $R'$ the condition \eqref{2.11} is valid.
Then $R'$ is a Riemannian $P$-tensor if the conditions
\eqref{2.12} and \eqref{2.13} are satisfy for $R'$, too. Since
$\nn P=0$ for the natural connection $\nn$ then \eqref{2.13} is
valid. The condition \eqref{2.12} for $R'$ is satisfied according
to \eqref{44''} if and only if the following equality is valid
\begin{equation}\label{5.1}
\mathop{\s} \limits_{x,y,z} \bigl\{\left(\nn_{x}
T\right)(y,z,w)+T(T(x,y),z,w)\bigr\}=0.
\end{equation}

Now, let us consider the case when $(M,P,g)$ belongs to the class
$\W_3$.
\begin{prop}\label{prop-5.1'}
If the curvature tensor of the canonical connection $\n'$ on a Riemannian almost product
$\W_3$-manifold $(M,P,g)$ is a Riemannian $P$-tensor, then the following identity for the torsion $T$ of $\n'$ is
valid
\begin{equation}\label{45'}
T(T(x,y),z,w)=0.
\end{equation}
\end{prop}
\begin{proof} We substitute $Pz$ for $z$ and $Pw$ for
$w$  in \eqref{5.1}. Hence, according to \eqref{4.5}, we obtain
\[
\begin{split}
&\left(\nn_{x} T\right)(y,z,w)-\left(\nn_{y}
T\right)(z,x,w)+\left(\nn_{Pz} T\right)(x,y,Pw)\\[4pt]
&+T(T(x,y),z,w)+T(T(y,Pz),x,w)+T(T(Pz,x),y,Pw)=0.
\end{split}
\]

We add the last equality to \eqref{5.1}, and substitute $Px$ for
$x$ and $Pw$ for $w$ in the result. Then, using \eqref{4.5}, we
get
\begin{equation}\label{5.2}
\begin{split}
&\left(\nn_{z} T\right)(x,y,z)-\left(\nn_{Pz}
T\right)(x,Py,w)\\[4pt]
&+2T(T(y,z),x,w)+2T(T(z,x),y,w)=0.
\end{split}
\end{equation}

In the latter equality we substitute $Py$, $Pz$ for $y$, $z$,
respectively, and we apply \propref{prop-4.1'}. We add the
obtained equality to \eqref{5.2} and this leads to \eqref{45'}.
\end{proof}

Let the curvature tensor of the canonical connection on
 a Riemannian almost product $\W_3$-manifold $(M,P,g)$
is a Riemannian $P$-tensor. Then \eqref{45'} is valid.
According to \eqref{45'}, \eqref{4.10} and the properties of $T$
from \propref{prop-4.1'}, we get
\[
F(Py,w,T(z,x))=-T(y,w,T(z,x)).
\]
Then, using \eqref{2} and \eqref{3.1} we obtain
\begin{equation}\label{5.3}
    g\bigl(T(x,z),T(y,w)+\left(\n_{Py}P\right)w\bigr)=0.
\end{equation}

Since, according to \eqref{32'}, we have
\[
T(y,w)=-\frac{1}{2}\big\{\left(\n_{y}P\right)Pw+\left(\n_{Py}P\right)w
\bigr\}
\]
and therefore the following equality is valid
\[
T(y,w)+\left(\n_{Py}P\right)w=-\frac{1}{2}\bigl\{\left(\n_{y}P\right)Pw-\left(\n_{Py}P\right)w
\bigr\}.
\]
By virtue of the latter two equalities and \eqref{5.3}, we arrive
at the following
\begin{thm}\label{thm-5.2}
If the curvature tensor of the canonical connection $\n'$ on a Riemannian almost product
$\W_3$-manifold $(M,P,g)$ is a Riemannian $P$-tensor, then the following identity for the torsion $T$ of $\n'$ is
valid
\[
g\bigl(\left(\n_{x}P\right)Pz+\left(\n_{Px}P\right)z,\left(\n_{Py}P\right)w-\left(\n_{y}P\right)Pw
\bigr)=0.
\]
\end{thm}

%%%%%%%%%%%%%%%%    6

\section{Canonical connection with parallel torsion on a Riemannian almost product  $\W_3$-manifold }

In this section we consider a canonical connection $\nn$ with
parallel torsion $T$ with respect to $\nn$ (\ie{} $\nn T=0$) on a Riemannian almost
product $\W_3$-manifold  $(M,P,g)$.

According to the Hayden theorem (\cite{Hay}) for any natural
connection we have
\begin{equation}\label{6.2}
    Q(x,y,z)=\frac{1}{2}\bigl\{
    T(x,y,z)-T(y,z,x)+T(z,x,y)\bigr\}.
\end{equation}

Combining this with \eqref{3.2}, \eqref{3.5}, \eqref{4.11}, leads
to the following
\begin{prop}\label{prop-6.1}
Let $\nn$ be a natural connection on a Riemannian almost product
$\W_3$-manifold  $(M,P,g)$. Then the tensors $T$, $Q$ and $F$ are
parallel or non-parallel at the same time with respect to $\nn$.

\end{prop}

Let $\nn$ be a natural connection with parallel torsion $T$ on a
Riemannian almost product manifold  $(M,P,g)$. According to
\eqref{6.2} we have $\nn Q=0$. Then, having in mind the formula
for the covariant derivative of $Q$, we obtain
\[%begin{equation}\label{6.3}
    xQ(y,z,w)-Q(\nn_x y,z,w)-Q(y,\nn_x z,w)-Q(y,z,\nn_x w)=0.
\]%\end{equation}
Applying the formula for the covariant derivative of $Q$ with
respect to $\n$ and equalities \eqref{3.1}, \eqref{3.2} and
\eqref{4.16}, we obtain the following
\begin{lem}\label{lem-6.2}
Let the canonical connection $\nn$ on a Riemannian almost product
manifold $(M,P,g)$ have a parallel torsion $T$. Then for the
curvature tensor $R'$ of $\nn$ is valid
\[%begin{equation}\label{6.4}
\begin{split}
    & R'(x,y,z,w)=R(x, y,z,w)+Q(T(x,y),z,w)\\[4pt]
    &\phantom{R'(x,y,z,w)=}+g(Q(y,z),Q(x,w))-g(Q(x,z),Q(y,w)).
\end{split}
\]%end{equation}

\end{lem}

Let $(M,P,g)$ be a Riemannian almost product $\W_3$-manifold whose
can\-onical connection $\nn$ has a parallel torsion $T$. Then,
according to \eqref{3.6}, \eqref{4.12} and \eqref{2}, we have
\[
    Q(T(x,y),z,w)=g(Q(z,w),T(x,y))-g(\left(\n_{Pw}P\right)z,T(x,y)).
\]

The last equality and \lemref{lem-6.2} imply the following
statement.
\begin{prop}\label{prop-6.3}
Let the canonical connection $\nn$ on a Riemannian almost product
$\W_3$-manifold $(M,P,g)$ have a parallel torsion $T$. Then for
the curvature tensor $R'$ of $\nn$ is valid
\[%begin{equation}\label{6.5}
\begin{split}
    & R'(x,y,z,w)=R(x, y,z,w)\\[4pt]
    &\phantom{R'(x,y,z,w)=}+g(Q(y,z),Q(x,w))-g(Q(x,z),Q(y,w))\\[4pt]
    &\phantom{R'(x,y,z,w)=}+g(Q(z,w),T(x,y))-g(\left(\n_{Pw}P\right)z,T(x,y)).
\end{split}
\]%\end{equation}
\end{prop}

\begin{thm}\label{thm-6.4}
Let the canonical connection $\nn$ on a Riemannian almost product
$\W_3$-manifold $(M,P,g)$ have a parallel torsion $T$. Then the
curvature tensor $R'$ of $\nn$ is a Riemannian $P$-tensor.
\end{thm}
\begin{proof} Let the canonical connection $\nn$ on a
Riemannian almost product $\W_3$-manifold $(M,P,g)$ have a
parallel torsion $T$, i.e. $\n'T=0$. Then from \eqref{44''} we
have
\begin{equation}\label{51}
\begin{split}
    & R'(x,y,z,w)+R'(y,z,x,w)+R'(z,x,y,w)\\[4pt]
&=T(T(x,y),z,w)+T(T(y,z),x,w)+T(T(z,x),y,w).
\end{split}
\end{equation}
We substitute $z \rightarrow Pz$ and $w \rightarrow Pw$ in
\eqref{51}.  After that, using property \eqref{2.13} for $R'$ and
the properties of $T$ from \propref{prop-4.1'}, we obtain
\[
\begin{split}
    & R'(x,y,z,w)+R'(y,Pz,Px,w)+R'(Pz,x,Py,w)\\[4pt]
&=-T(T(x,y),z,w)+T(T(y,z),x,w)+T(T(z,x),y,w).
\end{split}
\]
We subtract the latter equality from \eqref{51} and get
\begin{equation}\label{52}
\begin{split}
    & R'(z,x,y,w)+R'(y,z,x,w)-R'(Pz,x,Py,w)-R'(y,Pz,Px,w)\\[4pt]
&=2T(T(x,y),z,w).
\end{split}
\end{equation}
Bearing in mind the  properties of $R'$  and $T$, from
\eqref{52} we obtain
\[
\begin{split}
    & R'(z,x,y,w)-R'(z,Px,Py,w)-R'(Pz,x,Py,w)+R'(Pz,Px,y,w)=0,\\[4pt]
    & R'(z,x,y,w)-R'(z,Px,Py,w)-R'(Pz,Px,y,w)+R'(Pz,x,Py,w)=0,\\[4pt]
\end{split}
\]
which give
\begin{equation}\label{53}
     R'(z,x,y,w)=R'(z,Px,Py,w).
\end{equation}
Equality \eqref{53} implies
\begin{equation}\label{54}
     R'(z,x,y,w)=R'(Pz,Px,y,w)=R'(Pz,x,Py,w).
\end{equation}
Applying  \eqref{53} and \eqref{54} in  \eqref{52}, we obtain
\eqref{45'}. Then from  \eqref{51} it follows condition
\eqref{2.12} for $R'$, i.e. $R'$ is a Riemannian $P$-tensor.
\end{proof}

%%%%%%%%%%%%%%%%%%%%%%%%

\section{Example}

\subsection{A Lie group $G$ as a Riemannian almost product $\W_3$-manifold $(G,P,g)$}

Let $G$ be a 4-dimensional real connected Lie group and $\g$ be
its Lie algebra with a basis $\{X_1,X_2,X_3,X_4\}$.

We introduce a structure $P$ and left invariant metric $g$ as
follows
\begin{equation}\label{1*}
    PX_1=X_3,\quad PX_2=X_4,\quad PX_3=X_{1},\quad PX_4=X_{2},
\end{equation}
\begin{equation}\label{2*}
    g(X_i,X_j)=
\begin{cases}
\begin{array}{rl}
1, \quad & i=j;\\
0, \quad & i\neq j.
\end{array}
\end{cases}
\end{equation}
Obviously, $P^2X_i=X_i$, $g(PX_i,PX_j)=g(X_i,X_j)$ and then $(G,P,g)$ is a Riemannian almost product manifold with
$\tr{P}=0$.

\begin{thm}\label{thm-7.1}
If $(G,P,g)$ has a Killing associated metric $\tilde{g}$, i.e.
\begin{equation}\label{3*}
    g\left([X_i,X_j],PX_k\right)+g\left([X_i,X_k],PX_j\right)=0,
\end{equation}
then $(G,P,g)$ is a Riemannian almost product $\W_3$-manifold.
\end{thm}
\begin{proof} The well-known formula
\[
\begin{split}
    2g\bigl(\n_{X_i} X_j,X_k\bigr) &
    = X_ig\left(X_j,X_k\right)+X_jg\left(X_i,X_k\right)-X_kg\left(X_i,X_j\right)\\[4pt]
    &+g\left([X_i,X_j],X_k\right)+g\left([X_k,X_i],X_j\right)+g\left([X_k,X_j],X_i\right)
\end{split}
\]
and \eqref{2*} imply
\begin{equation}\label{4*}
\begin{split}
    2g\bigl(\n_{X_i} X_j,X_k\bigr)
    =
    g\left([X_i,X_j],X_k\right)&+g\left([X_k,X_i],X_j\right)\\[4pt]
    &+g\left([X_k,X_j],X_i\right).
\end{split}
\end{equation}
Since \eqref{3*} the following equalities are valid
\[
\begin{split}
&g\left([X_k,X_i],X_j\right)=g\left(P[X_i,PX_j],X_k\right),
\\[4pt]
&g\left([X_k,X_j],X_i\right)=-g\left(P[PX_i,X_j],X_k\right).
\end{split}
\]
Then from \eqref{4*} we get
\begin{equation}\label{5*}
        \n_{X_i}
        X_j=\frac{1}{2}\left\{[X_i,X_j]+P[X_i,PX_j]-P[PX_i,X_j]\right\},
\end{equation}
which implies
\begin{equation}\label{5**}
    \left(\n_{X_i}
    P\right)X_j=\frac{1}{2}\left\{[PX_i,X_j]-P[PX_i,PX_j]\right\}.
\end{equation}
Hence, according to \eqref{2}, we obtain
\begin{equation}\label{6*}
    F\left(X_i,X_j,X_k\right)=\frac{1}{2}\left\{g\bigl([PX_i,X_j],X_k\bigr)
    +g\bigl([PX_i,X_k],X_j\bigr)\right\}.
\end{equation}
Equality \eqref{6*} implies $\mathop{\mathfrak{S}}\limits_{i,j,k}
F(X_i,X_j,X_k)=0$. Therefore, $(G,P,g)$ is a Riemannian almost
product $\W_3$-manifold. \end{proof}

Let $(G,P,g)$ have a Killing associated metric $\tilde{g}$. Then,
according to \thmref{thm-7.1}, the manifold $(G,P,g)$ is a
Riemannian almost product $\W_3$-manifold. Applying condition
\eqref{3*} and the Jacobi identity for the commutators
$[X_i,X_j]$, we obtain
\begin{equation}\label{7*}
\begin{array}{ll}
[X_1,X_2]=\lm_1 X_1+\lm_2 X_2,\quad &
[X_1,X_3]=\lm_3 X_2-\lm_1 X_4,\\[4pt]
[X_1,X_4]=-\lm_3 X_1-\lm_2 X_4,\quad &
[X_2,X_3]=\lm_4 X_2+\lm_1 X_3,\\[4pt]
[X_2,X_4]=-\lm_4 X_1+\lm_2 X_3,\quad &
[X_3,X_4]=\lm_3 X_3+\lm_{4} X_4,\\[4pt]
\end{array}
\end{equation}
where $\lm_i\in \R$ $(i=1,2,3,4)$.

Vice verse, if equalities \eqref{7*} are valid for a Riemannian
almost product manifold $(G,P,g)$, then we verify directly that
the  Jacobi identity for the commutators $[X_i,X_j]$ is satisfied
and  the associated metric $\tilde{g}$ is Killing.

Hence, it is valid the following
\begin{thm}\label{thm-7.2}
The manifold $(G,P,g)$ is a Riemannian almost product
$\W_3$-manifold with a Killing associated metric $\tilde{g}$ if
and only if the Lie algebra $\g$ is determined by conditions
\eqref{7*}.
\end{thm}

Further, $(G,P,g)$ will stand for the Riemannian almost product
$\W_3$-manifold determined by  conditions \eqref{7*}.

\subsection{Some geometrical characteristics of the manifold $(G,P,g)$}

According to \eqref{6*}, \eqref{1*}, \eqref{2*} and \eqref{7*}, we
get the non-zero components \\ $F_{ijk}=F(X_i,X_j,X_k)$ of the
tensor $F$:
\begin{equation}\label{8*}
\begin{split}
2F_{114}=-2F_{123}=2F_{312}=-2F_{334}=-F_{411}=F_{433}=\lm_1,\\[4pt]
-2F_{223}=2F_{241}=F_{322}=-F_{344}=-2F_{412}=2F_{434}=\lm_2,\\[4pt]
-2F_{112}=2F_{134}=F_{211}=-F_{233}=-2F_{314}=2F_{332}=\lm_3,\\[4pt]
-F_{122}=F_{144}=2F_{221}=-2F_{234}=-2F_{414}=2F_{432}=\lm_4.\\[4pt]
\end{split}
\end{equation}
The rest of the non-zero components are obtained by the property
$F_{ijk}=F_{ikj}$.

By virtue of \eqref{5**} and \eqref{7*} we compute the components
of $\n P$:
\begin{equation}\label{56*}
\begin{array}{l}
        2\left(\n_{X_1} P\right)X_1=-2\left(\n_{X_3} P\right)X_3=-\lm_3 X_2 + \lm_1 X_4,\\[4pt]
        2\left(\n_{X_2} P\right)X_2=-2\left(\n_{X_4} P\right)X_4= \lm_4 X_1 - \lm_2 X_3,\\[4pt]
        2\left(\n_{X_1} P\right)X_3=-2\left(\n_{X_3} P\right)X_1=-\lm_1 X_2 + \lm_3 X_4,\\[4pt]
        2\left(\n_{X_2} P\right)X_4=-2\left(\n_{X_4} P\right)X_2= \lm_2 X_1-\lm_4 X_3,\\[4pt]
        2\left(\n_{X_1} P\right)X_2=-\lm_3 X_1 -2\lm_4 X_2 -\lm_1 X_3,\\[4pt]
        2\left(\n_{X_1} P\right)X_4=\lm_1 X_1 +\lm_3 X_3 +2\lm_4 X_4,\\[4pt]
        2\left(\n_{X_2} P\right)X_1=2\lm_3 X_1 +\lm_4 X_2 +\lm_2 X_4,\\[4pt]
        2\left(\n_{X_2} P\right)X_3=-\lm_2 X_2 -2\lm_3 X_3 -\lm_4 X_4,\\[4pt]
        2\left(\n_{X_3} P\right)X_2=\lm_1 X_1 +2\lm_2 X_2 +\lm_3 X_3,\\[4pt]
        2\left(\n_{X_3} P\right)X_4=-\lm_3 X_1 -\lm_1 X_3 -2\lm_2 X_4,\\[4pt]
        2\left(\n_{X_4} P\right)X_1=-2\lm_1 X_1 -\lm_2 X_2 -\lm_4 X_4,\\[4pt]
        2\left(\n_{X_4} P\right)X_3=\lm_4 X_2 +2\lm_1 X_3 +\lm_2 X_4.
\end{array}
\end{equation}

Using these components and \eqref{5}, we obtain the square norm of
$\n P$:
\begin{equation}\label{8**}
    \nP=4\left(\lm_1^2+\lm_2^2+\lm_3^2+\lm_4^2\right).
\end{equation}

The components of the Levi-Civita connection $\n$ we compute by
\eqref{5*} and \eqref{7*}:
\begin{equation}\label{9*}
    \begin{array}{l}
\n_{X_1} X_1=-\n_{X_3} X_3=-\lm_1 X_2+\lm_3 X_4,\\[4pt]%
\n_{X_2} X_2=-\n_{X_4} X_4=\lm_2 X_1-\lm_4 X_3,\\[4pt]%
2\n_{X_1} X_3=-2\n_{X_3} X_1=\lm_3 X_2-\lm_1 X_4,\\[4pt]%
2\n_{X_2} X_4=-2\n_{X_4} X_2=-\lm_4 X_1+\lm_2 X_3,\\[4pt]
    2\n_{X_1} X_2=2\lm_1 X_1-\lm_3 X_3+\lm_4 X_4,\\[4pt]%
2\n_{X_1} X_4=-2\lm_3 X_1-\lm_4 X_2+\lm_1 X_3,\\[4pt]
2\n_{X_2} X_1=-2\lm_2 X_2-\lm_3 X_3+\lm_4 X_4,\\[4pt]%
2\n_{X_2} X_3=\lm_3 X_1+2\lm_4 X_2-\lm_2 X_4,\\[4pt]
2\n_{X_3} X_2=\lm_3 X_1-2\lm_1 X_3-\lm_2 X_4,\\[4pt]%
2\n_{X_3} X_4=-\lm_1 X_1+\lm_2 X_2+2\lm_3 X_3,\\[4pt]
2\n_{X_4} X_1=-\lm_4 X_2+\lm_1 X_3+2\lm_2 X_4,\\[4pt]%
2\n_{X_4} X_3=-\lm_1 X_1+\lm_2 X_2-2\lm_4 X_4.
    \end{array}
\end{equation}

By virtue of \eqref{9*} and \eqref{7*}, from the formula
\[
\begin{split}
    R(X_i,X_j,X_k,X_s)=g\left(\n_{X_i}\n_{X_j} X_k,X_s\right)&-g\left(\n_{X_j}\n_{X_i} X_k,X_s\right)\\[4pt]
        & -g\left(\n_{[X_i,X_j]} X_k,X_s\right)
\end{split}
\]
we get the following non-zero components
$R_{ijks}=R(X_i,X_j,X_k,X_s)$ of the tensor $R$:
%\begin{subequations}
\begin{equation}\label{10*}
\begin{array}{l}
    R_{1234}=\lm_1^2+\lm_2^2,\\[4pt]
    R_{1331}=\frac{1}{4}\left(\lm_3^2+\lm_1^2\right),\quad
    R_{2442}=\frac{1}{4}\left(\lm_4^2+\lm_2^2\right),\\[4pt]
    R_{1441}=\frac{1}{4}\left(\lm_4^2+\lm_1^2-4\lm_3^2-4\lm_2^2\right),\\[4pt]
    R_{2332}=\frac{1}{4}\left(\lm_3^2+\lm_2^2-4\lm_4^2-4\lm_1^2\right),\\[4pt]
    R_{3443}=\frac{1}{4}\left(\lm_1^2+\lm_2^2-4\lm_3^2-4\lm_4^2\right),\\[4pt]
    R_{1221}=\frac{1}{4}\left(\lm_3^2+\lm_4^2-4\lm_1^2-4\lm_2^2\right),\\[4pt]
    R_{1241}=R_{3243}=\frac{1}{4}\left(4\lm_2\lm_4+3\lm_1\lm_3\right),\\[4pt]
    R_{2132}=R_{4134}=\frac{1}{4}\left(4\lm_1\lm_3+3\lm_2\lm_4\right),\\[4pt]
    R_{1231}=R_{4234}=\frac{1}{4}\left(3\lm_1\lm_4-2\lm_2\lm_3\right),\\[4pt]
    R_{2142}=R_{3143}=\frac{1}{4}\left(3\lm_2\lm_3-2\lm_1\lm_4\right),\\[4pt]
    R_{1341}=R_{2342}=\frac{1}{4}\left(3\lm_3\lm_4-2\lm_1\lm_2\right),\\[4pt]
    R_{3123}=R_{4124}=\frac{1}{4}\left(3\lm_1\lm_2-2\lm_3\lm_4\right).
\end{array}
\end{equation}
%\end{subequations}
The rest of the non-zero components are obtained by the properties
\[
R_{ijks}=R_{ksij},\qquad R_{ijks}=-R_{jiks}=-R_{ijsk}.
\]

Using \eqref{10*} for the non-zero components
$\rho_{ij}=\rho(X_i,X_j)$ of the Ricci tensor $\rho$ we compute:
\begin{equation}\label{10**}
\begin{array}{ll}
    \rho_{11}=\frac{1}{2}\left(-\lm_3^2-\lm_1^2+\lm_4^2-4\lm_2^2\right),\; &
    \rho_{34}=\frac{1}{2}\left(3\lm_3\lm_4-2\lm_1\lm_2\right),\\[4pt]
    \rho_{22}=\frac{1}{2}\left(-\lm_4^2-\lm_2^2+\lm_3^2-4\lm_1^2\right),\; &
    \rho_{14}=\frac{1}{2}\left(3\lm_2\lm_3-2\lm_1\lm_4\right),\\[4pt]
    \rho_{33}=\frac{1}{2}\left(-\lm_1^2+\lm_2^2-\lm_3^2-4\lm_4^2\right),\; &
    \rho_{23}=\frac{1}{2}\left(3\lm_1\lm_4-2\lm_2\lm_3\right),\\[4pt]
    \rho_{44}=\frac{1}{2}\left(-\lm_2^2+\lm_1^2-\lm_4^2-4\lm_3^2\right),\; &
    \rho_{24}=\frac{1}{2}\left(4\lm_2\lm_4+3\lm_1\lm_3\right),\\[4pt]
    \rho_{12}=\frac{1}{2}\left(-3\lm_1\lm_2+2\lm_3\lm_4\right),\; &
    \rho_{13}=\frac{1}{2}\left(4\lm_1\lm_3+3\lm_2\lm_4\right).
\end{array}
\end{equation}
The rest of the non-zero components are obtained by the property
$\rho_{ij}=\rho_{ji}$.

By \eqref{10**} we obtain the scalar curvature $\tau$ for the
connection $\n$:
\begin{equation}\label{11*}
    \tau=-\frac{5}{2}\left(\lm_1^2+\lm_2^2+\lm_3^2+\lm_4^2\right).
\end{equation}

For the Riemannian sectional curvatures of the $P$-invariant basis
2-planes \allowbreak $(X_1,X_3)$ and $(X_2,X_4)$, i.e. for the
invariant sectional curvatures of the basis 2-planes, we get
\begin{equation}\label{12*}
    k_{13}=\frac{1}{4}\left(\lm_1^2+\lm_3^2\right),\qquad
    k_{24}=\frac{1}{4}\left(\lm_2^2+\lm_4^2\right).
\end{equation}
The sectional curvatures of the rest of the basis 2-planes, i.e.
the anti-invariant sectional curvatures of the basis 2-planes,
are:
\begin{equation}\label{13*}
    \begin{array}{l}
    k_{12}=\frac{1}{4}\left(\lm_3^2+\lm_4^2-4\lm_1^2-4\lm_2^2\right),\\[4pt]
    k_{14}=\frac{1}{4}\left(\lm_1^2+\lm_4^2-4\lm_3^2-4\lm_2^2\right),\\[4pt]
    k_{23}=\frac{1}{4}\left(\lm_2^2+\lm_3^2-4\lm_4^2-4\lm_1^2\right),\\[4pt]
    k_{34}=\frac{1}{4}\left(\lm_1^2+\lm_2^2-4\lm_3^2-4\lm_4^2\right).
\end{array}
\end{equation}

Conditions \eqref{12*} and \eqref{13*} imply the following
\begin{prop}\label{prop-7.3}
The manifold $(G,P,g)$ has:
\begin{enumerate}\renewcommand{\labelenumi}{(\roman{enumi})}
    \item
    a constant invariant sectional curvature if and only if
$\lm_1^2-\lm_2^2+\lm_3^2-\lm_4^2=0$;
    \item
    a constant anti-invariant sectional curvature if and only if
$\lm_1^2=\lm_3^2$, $\lm_2^2=\lm_4^2$;
    \item
    a constant sectional curvature if and only if
$\lm_1^2=\lm_2^2=\lm_3^2=\lm_4^2$.
\end{enumerate}
\end{prop}

According to \eqref{7*}, \eqref{8*} and \eqref{11*}, it is valid
the following
\begin{prop}\label{prop-7.4}
The following propositions are equivalent:
\begin{enumerate}\renewcommand{\labelenumi}{(\roman{enumi})}
    \item
    $\lm_1=\lm_2=\lm_3=\lm_4=0$;
    \item
    the Lie algebra $\g$ is Abelian, i.e. $[X_i,X_j]=0$ $(i,j=1,2,3,4)$;
    \item
    $(G,P,g)$ is a Riemannian $P$-manifold, i.e. $F=0$;
    \item
    $(G,P,g)$ is a scalar flat manifold with respect to $\n$, i.e.
    $\tau=0$.
\end{enumerate}
\end{prop}

\subsection{The canonical connection on  $(G,P,g)$}

Further in our considerations we exclude the trivial case of
\propref{prop-7.4} for the Riemannian almost product
$\W_3$-manifold $(G,P,g)$.

By virtue of \eqref{4.7}, \eqref{8*} and \eqref{9*} for the
components of the canonical connection $\n'$ on  $(G,P,g)$ we
obtain:
\begin{equation}\label{14*}
    \begin{array}{l}
2\n'_{X_1} X_1=-2\n'_{X_3} X_3=-\lm_1 X_2+\lm_3 X_4,\\[4pt]%
2\n'_{X_2} X_2=-2\n'_{X_4} X_4=\lm_2 X_1-\lm_4 X_3,\\[4pt]%
2\n'_{X_1} X_2=-2\n'_{X_3} X_4=\lm_1 X_1-\lm_3 X_3,\\[4pt]%
2\n'_{X_1} X_3=-2\n'_{X_3} X_1=\lm_3 X_2-\lm_1 X_4,\\[4pt]%
2\n'_{X_1} X_4=-2\n'_{X_3} X_2=-\lm_3 X_1+\lm_1 X_3,\\[4pt]
2\n'_{X_2} X_1=-2\n'_{X_4} X_3=-\lm_2 X_2+\lm_4 X_4,\\[4pt]%
2\n'_{X_2} X_3=-2\n'_{X_4} X_1=\lm_4 X_2-\lm_2 X_4,\\[4pt]
2\n'_{X_2} X_4=-2\n'_{X_4} X_2=-\lm_4 X_1+\lm_2 X_3.
    \end{array}
\end{equation}

\begin{prop}\label{prop-7.5}
The curvature tensor of the canonical connection on $(G,P,g)$ is a Riemannian
$P$-tensor if and only if  the following conditions are valid:
\begin{equation}\label{15*}
    \lm_1=\varepsilon\lm_3,\quad \lm_2=\varepsilon\lm_4,\quad \varepsilon=\pm 1.
\end{equation}
\end{prop}
\begin{proof} Obviously the curvature tensor $R'$ of
the canonical connection is a Riemannian $P$-tensor if and only if
property \eqref{2.12} is satisfied for $R'$. Condition
\eqref{2.12} for $R'$ is valid if and only if the following
conditions are satisfied:
\[
\mathop{\s}_{1,2,3} R'_{123}=\mathop{\s}_{1,2,4}
R'_{124}=\mathop{\s}_{1,3,4} R'_{134}=\mathop{\s}_{2,3,4}
R'_{234}=0,
\]
where $R'_{ijk}$ are the components of $R'$.

Using \eqref{14*} and \eqref{7*}, we obtain:
\[
\begin{array}{l}
    R'_{123}=-\frac{1}{2}\left\{(\lm_1\lm_3+\lm_2\lm_4)X_2-(\lm_1^2+\lm_2^2)X_4\right\},\\[4pt]
    R'_{231}=-\frac{1}{2}\left\{(\lm_1\lm_3+\lm_2\lm_4)X_2+(\lm_1^2+\lm_4^2)X_4\right\},\\[4pt]
    R'_{312}=\frac{1}{2}\left\{(\lm_2\lm_3-\lm_1\lm_4)X_1+(\lm_1\lm_2-\lm_3\lm_4)X_3\right\},\\[4pt]
    R'_{124}=\frac{1}{2}\left\{(\lm_1\lm_3+\lm_2\lm_4)X_1-(\lm_1^2+\lm_2^2)X_3\right\},\\[4pt]
    R'_{241}=\frac{1}{2}\left\{(\lm_2\lm_3-\lm_1\lm_4)X_2-(\lm_1\lm_2-\lm_3\lm_4)X_4\right\},\\[4pt]
    R'_{412}=-\frac{1}{2}\left\{(\lm_1\lm_3+\lm_2\lm_4)X_1-(\lm_2^2+\lm_3^2)X_3\right\},\\[4pt]
    R'_{134}=\frac{1}{2}\left\{(\lm_3\lm_4-\lm_1\lm_2)X_1+(\lm_1\lm_4-\lm_2\lm_3)X_3\right\},\\[4pt]
    R'_{341}=-\frac{1}{2}\left\{(\lm_4^2-\lm_3^2)X_2+(\lm_2\lm_4+\lm_1\lm_3)X_4\right\},\\[4pt]
    R'_{413}=-\frac{1}{2}\left\{(\lm_2^2+\lm_3^2)X_2-(\lm_2\lm_4+\lm_1\lm_3)X_4\right\},\\[4pt]
    R'_{234}=-\frac{1}{2}\left\{-(\lm_1^2+\lm_4^2)X_1+(\lm_1\lm_3+\lm_2\lm_4)X_3\right\},\\[4pt]
    R'_{342}=-\frac{1}{2}\left\{(\lm_3^2+\lm_4^2)X_1-(\lm_1\lm_3+\lm_2\lm_4)X_3\right\},\\[4pt]
    R'_{423}=-\frac{1}{2}\left\{(\lm_3\lm_4-\lm_1\lm_2)X_2+(\lm_2\lm_3-\lm_1\lm_4)X_4\right\}.
\end{array}
\]
Hence it follows directly that condition \eqref{2.12} is valid for
$R'$ if and only if \eqref{15*} is satisfied.  \end{proof}

\propref{prop-7.3} and \propref{prop-7.5} imply the following
\begin{cor}\label{cor-7.6}
If the curvature tensor of the canonical connection on $(G,P,g)$ is a
Riemannian $P$-tensor, then $(G,P,g)$ is a manifold of constant
anti-invariant sectional curvature.
\end{cor}

\begin{prop}\label{prop-7.7}
The canonical connection on $(G,P,g)$ has a parallel torsion if
and only if the curvature tensor of this
connection is a Riemannian $P$-tensor.
\end{prop}
\begin{proof} By virtue of \eqref{32'} and \eqref{56*}
we obtain the non-zero components $T_{ij}=T(X_i,X_j)$ of the
torsion $T$ for the canonical connection $\n'$ on $(G,P,g)$:
\begin{equation}\label{17*}
    \begin{array}{l}
T_{12}=T_{34}=-\frac{1}{2}(\lm_1 X_1+\lm_2 X_2+\lm_3 X_3+\lm_4 X_4),\\[4pt]%
T_{14}=-T_{23}=\frac{1}{2}(\lm_3 X_1+\lm_4 X_2+\lm_1 X_3+\lm_2
X_4).
    \end{array}
\end{equation}
Obviously, the equality $T_{14}=PT_{12}$ is valid. Then, bearing
in mind that $\n'$ is a natural connection, we have
$\left(\n'_{X_i}T\right)(X_1,X_4)=-P\left(\n'_{X_i}T\right)(X_1,X_2)$.
From the latter condition and \eqref{17*} it is clear that the
condition $\n'T_{12}=0$ is sufficient for the obtaining of the
condition $\n'T=0$.

From \eqref{17*} we compute the following:
\[
    \begin{array}{l}
\left(\n'_{1}T\right)_{12}=-\frac{1}{4}\left\{(\lm_1\lm_2-\lm_3\lm_4)X_1
+(\lm_3^2-\lm_1^2)X_2+(\lm_1\lm_4-\lm_2\lm_3)X_3\right\},\\[4pt]%
\left(\n'_{2}T\right)_{12}=-\frac{1}{4}\left\{(\lm_2^2-\lm_4^2)X_1+(\lm_3\lm_4-\lm_1\lm_2)X_2
+(\lm_1\lm_4-\lm_2\lm_3)X_4\right\},\\[4pt]%
\left(\n'_{3}T\right)_{12}=-\frac{1}{4}\left\{(\lm_2\lm_3-\lm_1\lm_4)X_1
+(\lm_1\lm_2-\lm_3\lm_4)X_3+(\lm_1^2-\lm_3^2)X_4\right\},\\[4pt]%
\left(\n'_{4}T\right)_{12}=-\frac{1}{4}\left\{(\lm_2\lm_3-\lm_1\lm_4)X_2
+(\lm_4^2-\lm_2^2)X_3+(\lm_1\lm_2-\lm_3\lm_4)X_4\right\},
    \end{array}
\]
where
$\left(\n'_{i}T\right)_{12}=\left(\n'_{X_i}T\right)(X_1,X_2)$.
Hence $\n'T_{12}=0$ if and only if \eqref{15*} is valid, \ie{} if
and only if the curvature tensor of $\n'$ is a Riemannian
$P$-tensor.
\end{proof}

From \thmref{thm-4.4}, equalities \eqref{8**} and \eqref{11*}, we obtain the following
\begin{prop}\label{prop-7.8}
The manifold $(G,P,g)$ has negative scalar curvatures with respect
to the Levi-Civita connection $\n$ and the canonical connection
$\n'$, namely
\[
\tau=-\frac{5}{2}\left(\lm_1^2+\lm_2^2+\lm_3^2+\lm_4^2\right),\quad
\tau'=-2\left(\lm_1^2+\lm_2^2+\lm_3^2+\lm_4^2\right).
\]
\end{prop}

%%%%%%%%%%%%%%%%%%%%%%%%%%%%%


\begin{thebibliography}{10}

\bibitem{Li-1}
Lichnerowicz, A.: \emph{G\'{e}n\'{e}ralization de la
g\'{e}om\'{e}trie k\"ahlerienne globale}. Coll. de G\'{e}om. diff.
Louvain \textbf{16}, no.~2, 99-122 (1955)

\bibitem{Li-2}
Lichnerowicz, A.: \emph{Un th\'{e}or\`{e}me sur les espaces
homog\`{e}nes complexes}. Arch. Math. \textbf{5},  207-215 (1954)

\bibitem{GrayBNV}
Gray, A., Barros, M., Naveira, A., Vanheke, L.: \emph{The Chern
numbers of holomorphic vector bundles and formally holomorphic
connections of complex vector bundles over almost complex
manifolds}. P. reine andew. Math. \textbf{314},  84-98 (1980)

\bibitem{Mi}
Mihova, V.: \emph{Cannonical connections and the cannonical
conformal group on a Riemannian almost product manifold}. Serdica
Math. P. \textbf{15}, 351-358 (1989)

\bibitem{Ya}
Yano, K.: \emph{Differential geometry of complex and almost complex
spaces}, Pergamon press (1965)

\bibitem{Nav}
Naveira, A.~M.: \emph{A classification of Riemannian almost product
manifolds}. Rend. Math. \textbf{3}, 577-592 (1983)

\bibitem{Sta-Gri}
Staikova, M., Gribachev, K.: \emph{Canonical connections and their
conformal invariants on Riemannian P-manifolds}. Serdica Math. P.
\textbf{18}, 150-161 (1992)

\bibitem{Mek1}
Mekerov, D.: \emph{On Riemannian almost product manifolds with
nonintegrable structure}. J. Geom. \textbf{89}, no.~1-2,
119-129 (2008)


\bibitem{KoNo} %
Kobayashi, S., Nomizu, K.: \emph{Foundations of differential
geometry}, vol.~1, Intersci. Publ. (1963).

\bibitem{Hay} %
Hayden, H.: \emph{Subspaces of a space with torsion}. Proc. London
Math. Soc. \textbf{34}, 27-50 (1934)

\end{thebibliography}
\end{document}